\documentclass[11pt]{article}

\usepackage[T1]{fontenc}
\usepackage{lmodern}
\usepackage[margin=1in]{geometry}
\usepackage{amsmath,amssymb,amsthm}
\usepackage{microtype}
\usepackage{xcolor}
\definecolor{refblue}{RGB}{0,70,170}
\definecolor{citegreen}{RGB}{0,120,70}
\usepackage[
  colorlinks=true,
  linkcolor=refblue,
  citecolor=citegreen,
  urlcolor=black
]{hyperref}

\newtheorem{theorem}{Theorem}[section]
\newtheorem{lemma}[theorem]{Lemma}

\numberwithin{equation}{section}

\newcommand{\R}{\mathbb{R}}
\newcommand{\E}{\mathbb{E}}
\newcommand{\Pcal}{\mathcal{P}}
\newcommand{\Scal}{\mathcal{S}}

\title{The exact dimensional threshold for Spearman rank-correlation compatibility}
\author{%
Ruodu Wang\\
{\small Department of Statistics and Actuarial Science}\\
{\small University of Waterloo}\\
{\small \href{mailto:wang@uwaterloo.ca}{\texttt{wang@uwaterloo.ca}}}
\and
Zhenyuan Zhang\\
{\small Department of Mathematics}\\
{\small Stanford University}\\
{\small \href{mailto:zzy@stanford.edu}{\texttt{zzy@stanford.edu}}}
}
\date{}

\begin{document}

\maketitle

\begin{abstract}
We show that, for a given dimension, the set of Spearman's rank correlation matrices 
and that of linear correlation matrices coincide if and only if the dimension is no larger than nine. 
For this, we construct an extreme rank-four counterexample in dimension ten and prove its incompatibility using  moment identities and Cauchy–Schwarz. Appending unit directions produces   counterexamples in every higher dimension. This, together with existing results, completes the dimensional classification and  settles a long-standing open question in quantitative risk management. \medskip

\end{abstract}

\section{Introduction}

The paper of Embrechts, McNeil, and Straumann \cite{EMS2002}
is a celebrated piece in  quantitative risk management (QRM), and it has been a constant source of inspiration for scientific  developments in the field over the past three decades (an early version of \cite{EMS2002} was dated 1999); many subsequent developments are included in the standard textbook \cite{MFE15}. 
Among other important insights, the authors of \cite{EMS2002} explicitly asked the following question in their Section~6.2 (and provided a positive answer in dimension $2$):
\begin{equation}
    \label{eq:main-Q}
    \mbox{Is every linear correlation matrix necessarily a Spearman's rank correlation matrix?}
\end{equation}

In this paper, we provide a full answer to \eqref{eq:main-Q} in every dimension. To explain this question and its relevance to QRM, 
we note that the dependence structure in a multivariate probabilistic model 
is a crucially important part of  quantitative modeling in QRM.
For tractability in applications, the dependence structure is often  summarized by a matrix of
pairwise association measures.  When the entries are Pearson's linear correlations,
the admissible matrices are exactly the real symmetric positive semidefinite
matrices with unit diagonal. 
When the entries are Spearman's rank correlations, which have many advantages over the linear correlations,
such a matrix 
becomes more delicate. 
It is straightforward to check that a Spearman's rank correlation matrix
must be a linear correlation matrix, and the question in \eqref{eq:main-Q} asks the converse question, which is much more difficult. 
An answer to \eqref{eq:main-Q} has useful implications in QRM. For instance, 
it tells us whether an empirical estimator of a rank correlation matrix  is guaranteed to be a valid one.

Let $d$ be  a positive integer, which is the underlying dimension of the problem.
A Spearman's rank correlation matrix is the linear correlation matrix of a random vector with uniform marginal distributions, and the set of $d\times d$ Spearman's rank correlation matrices is  denoted by $\mathcal S_d$. 
Indeed, if $(X_1,\ldots,X_d)$ has continuous marginal distribution
functions $F_1,\ldots,F_d$, then
\[
Z_i=\sqrt{12}\left(F_i(X_i)-\frac{1}{2}\right)
\]
is uniformly distributed on $[-\sqrt{3},\sqrt{3}]$, and the Spearman correlation of
$X_i$ and $X_j$ equals $\E[Z_iZ_j]$.  Hence, $R\in\Scal_d$ if and only if
$R$ is the covariance matrix of a random vector whose marginal
distributions are all uniform on $[-\sqrt{3},\sqrt{3}]$.  
 Let $\Pcal_d$ be the set of real symmetric
positive semidefinite $d\times d$ matrices with unit diagonal (i.e.,~linear correlation matrices). 
The compatibility question \eqref{eq:main-Q} is whether $\Scal_d=\Pcal_d$.

The answer to \eqref{eq:main-Q} was known to be affirmative in low dimensions.  For $d=2$, it was shown by a simple construction in  \cite{EMS2002}. 
As a major nontrivial result,  
Devroye and
Letac \cite{DL2015} proved that $\Scal_d=\Pcal_d$ for $d\leq9$.  A key
ingredient is the geometry of the extreme points of $\Pcal_d$: the results
of Ycart \cite{Ycart1985} and Grone, Pierce, and Watkins \cite{GPW1990}
imply that all such extreme points have rank at most three when $d\leq9$.
The uniform distribution on the sphere in $\R^3$ has uniform
one-dimensional projections, and its rotational invariance therefore
realizes every correlation matrix of rank at most three.  Convexity then
extends the conclusion to all of $\Pcal_d$.  The threshold $d=10$ is the
first dimension in which an extreme correlation matrix may have rank four,
and the spherical construction no longer applies.

Devroye and
Letac \cite{DL2015} 
conjectured that $\Scal_d=\Pcal_d$  fails 
for $d\ge 10$, although they did not find a proof. 
A first negative answer was provided by Wang, Wang, and Wang  \cite{WWW2019}, who showed   $\Scal_d\neq\Pcal_d$ for every $d\geq12$. 
They 
developed a rank-decomposition characterization that
makes the above obstruction precise.  If $R=AA^{\mathsf T}$ has
rank $k$, then $R$ is Spearman-compatible exactly when there exists a
centered random vector $\mathbf{V}\in\R^k$ with identity covariance such that every
component of $A\mathbf{V}$ is uniform on $[-\sqrt{3},\sqrt{3}]$.  Using twelve
carefully chosen directions in $\R^4$, they showed that no such $\mathbf{V}$ exists
for a particular $12\times12$ correlation matrix.   Their contradiction relies on
the symmetry of the twelve directions, which converts their second- and
fourth-power sums into functions of $\lVert \mathbf{V}\rVert$ alone.  They noted
that this symmetry is unavailable with ten or eleven directions and left
those two dimensions open; they also explicitly anticipated that $\Scal_d=\Pcal_d$ fails for $d=10$ or $11$.
In a more recent paper, McNeil, Ne\v{s}lehov\'a, and Smith \cite{MNS2022} noted that whether $\Scal_d=\Pcal_d$  for $d=10,11$ remained open. They solved an analogous problem for Kendall's tau, another commonly used coefficient of rank correlation, 
by showing that the Kendall's tau
matrices coincide with linear correlation matrices
if and only if $d\le2$.

In this paper, we settle the remaining cases of Spearman compatibility in full.  We construct an explicit
rank-four matrix $R_{10}\in\Pcal_{10}$ and prove that it is not
Spearman-compatible.  Appending an identity block then gives
counterexamples in every larger dimension, including dimension eleven.
Combined with the positive result for $d\leq9$, this completes the
dimensional classification.

\begin{theorem}\label{thm:classification}
For every positive integer $d$, one has $\Scal_d=\Pcal_d$ if and only if
$d\leq 9$.
\end{theorem}

The construction replaces the global symmetry used in the
twelve-dimensional example by two complementary families of directions.
Four directions form a regular tetrahedron in the last three coordinates,
and six directions occur in three pairs coupling the first coordinate to
one of the other coordinates.  If the corresponding matrix
$R_{10}=CC^{\mathsf T}$ were Spearman-compatible, the rank-decomposition
criterion would produce a centered four-dimensional random vector with ten
specified uniform projections.  Symmetrizing its law eliminates mixed
moments containing an odd power.  The fourth moments along the tetrahedral
directions then determine the relevant moments of the last three
coordinates, while the paired directions link these moments to the first
coordinate.  The resulting identities violate Cauchy--Schwarz.  Thus, in
contrast to the earlier twelve-direction argument, a fully spherical
fourth-moment identity is not needed.

Section~\ref{sec:construction} defines the ten directions, the matrix $C$,
and the resulting correlation matrix $R_{10}=CC^{\mathsf T}$.
Section~\ref{sec:noncompatibility} recalls the uniform-marginal reduction,
proves that this matrix is not Spearman-compatible, and completes the
proof of Theorem~\ref{thm:classification}, including the $11\times11$
case.

\section{Construction of the 10-by-10 counterexample}
\label{sec:construction}
In this section, we construct an explicit matrix
$R_{10}\in\Pcal_{10}$; Theorem~\ref{thm:d10} below will show that
$R_{10}\notin\Scal_{10}$.  Let \(\mathbf{e}_1,\ldots,\mathbf{e}_4\) denote the standard basis vectors of \(\mathbb R^4\). We say that 
$\mathbf{s}=(s_1,s_2,s_3)\in\{-1,1\}^3$ is \textit{admissible} if $s_1s_2s_3=1$.  For each
admissible $\mathbf{s}$, and for $j\in\{1,2,3\}$ and
$\varepsilon\in\{-1,1\}$, consider the unit vectors
\begin{equation}\label{eq:directions}
\mathbf{a}_{\mathbf{s}}=\frac{1}{\sqrt{3}}(0,s_1,s_2,s_3),
\qquad
\mathbf{b}_{j,\varepsilon}
=\frac{1}{\sqrt{3}}\mathbf{e}_1
+\varepsilon\sqrt{\frac{2}{3}}\,\mathbf{e}_{j+1}.
\end{equation}
 Let $C$ be the
$10\times4$ matrix whose rows are the four vectors
$\mathbf{a}_{\mathbf{s}}^{\mathsf T}$ followed by the six vectors
$\mathbf{b}_{j,\varepsilon}^{\mathsf T}$:
\begin{equation}\label{eq:C-explicit}
C=\frac{1}{\sqrt{3}}
\begin{pmatrix}
0& 1& 1& 1\\
0& 1&-1&-1\\
0&-1& 1&-1\\
0&-1&-1& 1\\
1& \sqrt{2}&0&0\\
1&-\sqrt{2}&0&0\\
1&0& \sqrt{2}&0\\
1&0&-\sqrt{2}&0\\
1&0&0& \sqrt{2}\\
1&0&0&-\sqrt{2}
\end{pmatrix}.
\end{equation}
Clearly, $\operatorname{rank}(C)=4$. Define the $10\times10$ matrix
\begin{equation}\label{eq:R10}
R_{10}=CC^{\mathsf T}.
\end{equation}
For every $\mathbf{x}\in\R^{10}$,
$\mathbf{x}^{\mathsf T}R_{10}\mathbf{x}=\|C^{\mathsf T}\mathbf{x}\|^2\geq0$, and the diagonal
entries of $R_{10}$ are one because the rows of $C$ are unit vectors.
Thus, $R_{10}\in\Pcal_{10}$.  Moreover,
$\ker(R_{10})=\ker(C^{\mathsf T})$, so
$\operatorname{rank}(R_{10})=\operatorname{rank}(C)=4$.

\section{Failure of Spearman compatibility}
\label{sec:noncompatibility}

We first record the characterization of the set $\Scal_d$ of  Wang, Wang, and
Wang \cite[Theorem~2.2]{WWW2019}.

\begin{lemma}\label{lem:reduction}
Suppose that $C\in\R^{d\times k}$ has rank $k$, every row
$\mathbf{c}_i^{\mathsf T}$ of $C$ is a unit vector, and $R=CC^{\mathsf T}$.
Then $R\in\Scal_d$ if and only if there is a random vector
$\mathbf{V}\in\R^k$ such that
$$
\E[\mathbf{V}]=0,\qquad \operatorname{Cov}(\mathbf{V})=I_k,\qquad
\mathbf{c}_i^{\mathsf T}\mathbf{V}\sim\operatorname{Unif}[-\sqrt{3},\sqrt{3}]
$$
for every $i\in\{1,\ldots,d\}$.
\end{lemma}

\begin{theorem}\label{thm:d10}
The matrix $R_{10}$ defined in \eqref{eq:R10} does not belong to
$\Scal_{10}$.
\end{theorem}

\begin{proof}
Assume for contradiction that $R_{10}\in\Scal_{10}$.
Lemma~\ref{lem:reduction} then gives a random vector
$\mathbf{V}=(X,Y_1,Y_2,Y_3)^{\mathsf T}$ satisfying
\begin{equation}\label{eq:V-properties}
\E[\mathbf{V}]=0,\qquad \operatorname{Cov}(\mathbf{V})=I_4,
\end{equation}
such that the projections of $\mathbf{V}$ along all ten vectors in
\eqref{eq:directions} are uniform on
$[-\sqrt{3},\sqrt{3}]$. We split the rest of the proof into three steps.

\medskip
\noindent\emph{Step 1: Reducing to the symmetric case.}
Let $G$ be the finite group of matrices
$h=\operatorname{diag}(\eta,P)$, where $\eta\in\{-1,1\}$ and $P$ is a
signed $3\times3$ permutation matrix, meaning that each row and each
column has exactly one nonzero entry and that entry is $1$ or $-1$.
Observe that if $h\in G$,  $h\mathbf{V}$ satisfies \eqref{eq:V-properties} and the same ten uniform
projection constraints as $\mathbf{V}$.  By replacing $\mathbf{V}$ with $H\mathbf{V}$, where $H$
is uniform on $G$ and independent of $\mathbf{V}$, we may assume that
the law of $\mathbf{V}$ is $G$-invariant.

Since $G$ contains every permutation of $Y_1,Y_2,Y_3$, the variables $Y_j$ are
exchangeable. Moreover, $G$ contains each marginal sign change, so the expectation of any monomial in $X,Y_1,Y_2,Y_3$ containing an odd
power of at least one coordinate is zero.

\medskip
\noindent\emph{Step 2: Moment identities and inequalities.}
A random variable $U\sim\operatorname{Unif}[-\sqrt{3},\sqrt{3}]$ satisfies
$\E[U^4]=9/5$.  For every admissible $\mathbf{s}=(s_1,s_2,s_3)$, Step 1 therefore gives
\begin{equation}\label{eq:a-fourth}
\frac{27}{5}
=3\E[(\mathbf{a}_{\mathbf{s}}^{\mathsf T}\mathbf{V})^4]
=\frac{1}{3}\E[(s_1Y_1+s_2Y_2+s_3Y_3)^4]
=\E[Y_1^4]+6\E[Y_1^2Y_2^2].
\end{equation}
Similarly, for $\varepsilon\in\{\pm 1\}$,
\begin{equation}\label{eq:b-fourth}
\frac{81}{5}
=9\E[(\mathbf{b}_{1,\varepsilon}^{\mathsf T}\mathbf{V})^4]
=\E[(X+\varepsilon\sqrt{2}Y_1)^4]
=\E[X^4]+12\E[X^2Y_1^2]+4\E[Y_1^4].
\end{equation}
Set
$T=(Y_1^2+Y_2^2+Y_3^2)/\sqrt{3}$.  By \eqref{eq:V-properties},
$\E[T]=\sqrt{3}$, so exchangeability
gives
\begin{equation}\label{eq:T-variance}
\operatorname{Var}(T)
=\E[T^2]-\E[T]^2
=\frac{1}{3}\E[(Y_1^2+Y_2^2+Y_3^2)^2]-3
=\E[Y_1^4]+2\E[Y_1^2Y_2^2]-3.
\end{equation}
Combining \eqref{eq:a-fourth} and \eqref{eq:T-variance} yields
\begin{equation}\label{eq:Y-moments}
\E[Y_1^2Y_2^2]
=\frac{3}{5}-\frac{\operatorname{Var}(T)}{4},\qquad
\E[Y_1^4]
=\frac{9}{5}+\frac{3\operatorname{Var}(T)}{2}.
\end{equation}
Substituting \eqref{eq:Y-moments} into \eqref{eq:b-fourth} gives
\begin{equation}\label{eq:mixed-moment}
12\E[X^2Y_1^2]
=\frac{81}{5}-\E[X^4]-4\E[Y_1^4]
=9-\E[X^4]-6\operatorname{Var}(T).
\end{equation}
The left side of \eqref{eq:mixed-moment} is nonnegative, while \eqref{eq:V-properties} gives $\E[X^4]\geq\E[X^2]^2=1$.  Thus,
\begin{equation}\label{eq:moment-range}
1\leq\E[X^4]\leq9-6\operatorname{Var}(T).
\end{equation}

\medskip
\noindent\emph{Step 3: Cauchy--Schwarz.}
Since $\E[X^2]=1$ and $\E[T]=\sqrt{3}$, exchangeability and
\eqref{eq:mixed-moment} give
\begin{equation}\label{eq:covariance}
\begin{aligned}
\operatorname{Cov}(X^2,T)
&
=\frac{1}{\sqrt{3}}\sum_{j=1}^3\E[X^2Y_j^2]-\sqrt{3}=\sqrt{3}\bigl(\E[X^2Y_1^2]-1\bigr)=-\frac{\sqrt{3}}{12}
 \bigl(\E[X^4]+3+6\operatorname{Var}(T)\bigr).
\end{aligned}
\end{equation}
Using \eqref{eq:covariance} and  Cauchy--Schwarz, we have
\begin{equation}\label{eq:CS}
\bigl(\E[X^4]+3+6\operatorname{Var}(T)\bigr)^2
=48\operatorname{Cov}(X^2,T)^2
\leq48\operatorname{Var}(X^2)\operatorname{Var}(T)
=48\bigl(\E[X^4]-1\bigr)\operatorname{Var}(T).
\end{equation}
On the other hand, by \eqref{eq:moment-range}, we may apply the inequality $(p+q+4)^2>8pq,~p,q\geq 0,\,p+q\leq8$ with 
$p=\E[X^4]-1$ and $q=6\operatorname{Var}(T)$ to conclude that
\begin{equation}\label{eq:strict-reverse}
\bigl(\E[X^4]+3+6\operatorname{Var}(T)\bigr)^2
=(p+q+4)^2
>8pq
=48\bigl(\E[X^4]-1\bigr)\operatorname{Var}(T).
\end{equation}
This contradicts \eqref{eq:CS}.  Hence, the assumed random vector
cannot exist, and by using Lemma~\ref{lem:reduction} we get  
$R_{10}\notin\Scal_{10}$.
\end{proof}

\begin{proof}[Proof of Theorem~\ref{thm:classification}]
Devroye and Letac \cite{DL2015} proved that $\Scal_d=\Pcal_d$ for $d\leq9$.  Theorem~\ref{thm:d10} gives
$R_{10}\in\Pcal_{10}\setminus\Scal_{10}$, so $\Pcal_{10}\neq\Scal_{10}$. For every $d>10$, set
$R_d=\operatorname{diag}(R_{10},I_{d-10})$, which is clearly a correlation
matrix.  If $R_d$ belonged to $\Scal_d$, it
would be the covariance matrix of a random vector with uniform marginals in $\R^d$.
Its first ten coordinates would then have covariance matrix $R_{10}$ and
the same uniform marginals, so $R_{10}\in\Scal_{10}$, a
contradiction.  Hence, $\Scal_d\neq\Pcal_d$ for every $d\geq10$, which
completes the proof.
\end{proof}

\section*{Acknowledgments}

The counterexample presented in this paper was found by OpenAI's
GPT-5.6 Ultra.  The authors subsequently verified the construction and all
mathematical arguments in full. 
The authors wrote up the manuscript and  take full responsibility for its contents. GPT-5.6 further provided editing assistance.


\begin{thebibliography}{9}

\bibitem{DL2015}
L.~Devroye and G.~Letac,
\newblock Copulas with prescribed correlation matrix,
\newblock in C.~Donati-Martin, A.~Lejay, and A.~Rouault (eds.),
\emph{In Memoriam Marc Yor---S\'eminaire de Probabilit\'es XLVII},
Lecture Notes in Mathematics, vol.~2137,
Springer, Cham, 2015, pp.~585--601.
 

\bibitem{EMS2002}
P.~Embrechts, A.~J. McNeil, and D.~Straumann,
\newblock Correlation and dependence in risk management: properties and
pitfalls,
\newblock in M.~A.~H. Dempster (ed.),
\emph{Risk Management: Value at Risk and Beyond},
Cambridge University Press, Cambridge, 2002, pp.~176--223.
 

\bibitem{GPW1990}
R.~Grone, S.~Pierce, and W.~Watkins,
\newblock Extremal correlation matrices,
\newblock \emph{Linear Algebra and its Applications} 134 (1990), 63--70.
 

\bibitem{MFE15}
A.~J. McNeil, R.~Frey, and P.~Embrechts,
\newblock \emph{Quantitative Risk Management: Concepts, Techniques and Tools},
\newblock revised ed., Princeton University Press, Princeton, NJ, 2015.

\bibitem{MNS2022}
A.~J. McNeil, J.~G. Ne\v{s}lehov\'a, and A.~D. Smith,
\newblock On attainability of Kendall's tau matrices and concordance signatures,
\newblock \emph{Journal of Multivariate Analysis} 191 (2022), 105033.
 

\bibitem{WWW2019}
B.~Wang, R.~Wang, and Y.~Wang,
\newblock Compatible matrices of Spearman's rank correlation,
\newblock \emph{Statistics \& Probability Letters} 151 (2019), 67--72.
 

\bibitem{Ycart1985}
B.~Ycart,
\newblock Extreme points in convex sets of symmetric matrices,
\newblock \emph{Proceedings of the American Mathematical Society}
95 (1985), no.~4, 607--612.
 

\end{thebibliography}
\end{document}